\documentclass[accepted]{uai2025} 
                        
\usepackage[american]{babel}

\usepackage{natbib} 
\usepackage{mathtools} 
\usepackage{booktabs} 
\usepackage{tikz} 
\usepackage{amsmath}
\usepackage{amssymb}
\usepackage{mathtools}
\usepackage{amsthm}
\usepackage{algorithm}
\usepackage{algorithmic}

\theoremstyle{plain}
\newtheorem{theorem}{Theorem}[section]
\newtheorem{proposition}[theorem]{Proposition}
\newtheorem{lemma}[theorem]{Lemma}

\theoremstyle{definition}
\newtheorem{definition}[theorem]{Definition}

\newtheorem{remark}[theorem]{Remark}

\newcommand{\mbb}{\mathbb}
\newcommand{\mbf}{\mathbf}
\newcommand{\mcal}{\mathcal}

\usepackage{microtype}
\usepackage{graphicx}
\usepackage{subfigure,enumitem}
\usepackage{booktabs}

\def\namedlabel#1#2{\begingroup
    #2%
    \def\@currentlabel{#2}%
    \phantomsection\label{#1}\endgroup
}

\DeclareMathOperator{\Var}{Var}

\title{Minimax Optimal Nonsmooth Nonparametric Regression via Fractional Laplacian Eigenmaps}

\author[1]{Zhaoyang Shi}
\author[2]{Krishna Balasubramanian}
\author[2]{Wolfgang Polonik}
\affil[1]{%
    Department of Statistics\\
    Harvard University\\
    Cambridge, Massachusetts, USA
}
\affil[2]{%
    Department of Statistics\\
    University of California, Davis\\
    Davis, California, USA
}

\begin{document}
\maketitle

\begin{abstract}
  We develop minimax optimal estimators for nonparametric regression methods when the true regression function lies in an $L_2$-fractional Sobolev space with order $s\in (0,1)$. This function class is a Hilbert space lying between the space of square-integrable functions and the first-order Sobolev space consisting of differentiable functions. It contains fractional power functions, piecewise constant or piecewise polynomial functions and bump function as canonical examples. We construct an estimator based on performing Principal Component Regression using Fractional Laplacian Eigenmaps and show that the in-sample mean-squared estimation error of this estimator is of order $n^{-\frac{2s}{2s+d}}$, where $d$ is the dimension, $s$ is the order parameter and $n$ is the number of observations. We next prove a minimax lower bound of the same order, thereby establishing that no other estimator can improve upon the proposed estimator, up to context factors. We also provide preliminary empirical results validating the practical performance of the developed estimators.
\end{abstract}

\section{Introduction}
Laplacian based nonparametric regression is a widely used approach in machine learning that leverages the Laplacian Eigenmaps algorithm to perform regression tasks without relying on explicit parametric models. The nonparametric nature of the approach makes it flexible and adaptable to data generating processes without imposing strict assumptions about the functional form of the relationship between the response and the covariates. Existing theoretical studies of this approach are restricted to establishing minimax rates of convergence and adaptivity properties when the true regression function lies in Sobolev spaces; see Section~\ref{sec:litrev} for details. Such spaces are inherently smooth in nature and exclude important function classes in nonparametric statistics, such as piecewise constant or piecewise polynomial functions, bump functions and other such nonsmooth function classes. 

In this work, using the framework of fractional Laplacians, we propose a novel approach called Principal Component Regression using Fractional Laplacian Eigenmaps (PCR-FLE) for nonsmooth and nonparametric regression. The PCR-FLE algorithm generalizes the PCR-LE algorithm by~\citet{green2021minimax} and the PCR-WLE algorithm by~\citet{shi2023adaptive}, and is designed to naturally handle the case when the true regression function lies in an $L_2$-fractional Sobolev space $H^{s}(\mcal{X})$ (see Definition~\ref{def:fractionalsobolev}). Specifically, consider the following regression model, $Y_{i}=f(X_{i})+\varepsilon_{i}$, for $i=1,\ldots,n,$ where $f:\mathcal{X}\to \mathbb{R}$, $ f \in H^{s}(\mcal{X})$ for $s\in(0,1)$, $X_{i} \overset{\text{i.i.d.}}{\sim} g$, where $g$ is a density on $\mcal{X}\subset \mathbb{R}^d$, and $\varepsilon_{i}\overset{\text{i.i.d.}}{\sim} N(0,1)$ is the noise (independent of the $X_{i}$'s). The goal is to estimate the regression function $f$ given pairs of observations $(X_{1}, Y_{1}),\ldots, (X_{n}, Y_{n})$. The proposed PCR-FLE algorithm proceeds by estimating the eigenvalues and the eigenfunction of the fractional Laplacian operator (see~\eqref{eq:flap}) based on the eigenvalues and the eigenvectors of the $\epsilon$-graph constructed from the samples $\{X_i\}_{i=1}^n$, and projecting the response vector onto the top-$K$ eigenvectors. 

For this procedure, we make the following technical contributions in this work:
\begin{itemize}
    \item In Theorem \ref{mainthm}, we establish upper bounds for the in-sample mean-squared estimation error for the PCR-FLE algorithm, when the true regression function lies in $H^s({\mcal X})$, for $s\in(0,1)$ that hold with high-probability. As a part of the proof of our main results in Theorem \ref{mainthm}, we derive a concentration inequality of the discrete fractional Sobolev seminorm/energy to its continuum, which serves as an important quantity for other fractional Laplacian based machine learning algorithms.
    \item In Theorem \ref{thm:lowerbound}, we provide a minimax lower bound for integrated mean-squared estimation error when the truth is in $H^s({\mcal X})$, for $s\in(0,1)$, suggesting that the matching upper bounds in Theorem \ref{mainthm} are optimal. To the best of our knowledge, this is the first minimax optimality lower bounds result under random design over fractional Sobolev spaces. 
\end{itemize}
We provide preliminary simulations of the proposed approach validating the performance of PCR-FLE algorithm in estimating various nonsmooth functions in Section~\ref{eq:sim}. Our contributions underscore the importance of employing fractional graph Laplacians in nonsmooth nonparametric regression and lay a strong statistical groundwork for this technique.

\subsection{Literature review}\label{sec:litrev}

Graph Laplacians find extensive use in various data science applications, including feature learning and spectral clustering~\citep{weiss1999segmentation,shiNormalizedCutsImage2000a,ngSpectralClusteringAnalysis2001,von2007tutorial}. They are employed for tasks such as extracting heat kernel signatures for shape analysis~\citep{sun2009concise, andreuxAnisotropicLaplaceBeltramiOperators2015, dunson2021spectral}, reinforcement learning~\citep{mahadevan2007proto,wu2018the}, and dimensionality reduction \citep{belkin2003laplacian,coifman2006diffusion}, among others. Additional discussions can be found in works by \citet{belkin2006manifold, wang2015trend, chun2016eigenvector, hacquard2022topologically}.

Several papers in the recent past have focused on obtaining theoretical rates of convergence in the context of Laplacian operator estimation and related eigenvalue and/or eigenfunction estimation. Pointwise consistency under $\epsilon$-graphs have been studied by \citet{belkin2005towards,hein2005graphs, gine2006empirical,hein2007graph}; see references therein for more related works. Furthermore, \citet{trillos2018variational} derives consistency properties of spectral clustering methods through studying the above spectral convergence with no specific error estimates. Following them, \citet{shi2015convergence,garcia2020error,calder2022improved} derived rates of convergences of Laplacian eigenvalues and eigenvectors to population counterparts with explicit error estimates under both $\epsilon$-graphs and $k$-NN graph. Recently, \citet{hoffmann2022spectral} developed a framework for extending the above convergence results to a general Laplacian family, the weighted Laplacians and \citet{shi2023adaptive} provided additional theoretical results on the convergence of the weighted Laplacians. \citet{green2021minimaxsmoothing} considered Laplacian smoothing estimation and showed its minimax optimal rates in low-dimensional space. \citet{green2021minimax} proposed the principal components regression with the  Laplacian eigenmaps (PCR-LE) algorithm that achieves minimax optimal rates of nonparametric regression under uniform design. The PCR-LE algorithm was later generalized by \citet{shi2023adaptive} to the weighted Laplacians that include other commonly applied Laplacians such as the normalized Laplacian and the random walk Laplacian. Moreover, from a methodological perspective, \citet{rice1984bandwidth} investigated spectral series regression on Sobolev spaces, and \cite{trillos2022rates} applied the graph Poly-Laplacian smoothing to the regression problem.

The literature on both theoretical analysis and statistical applications fractional (graph) Laplacians is still in its infancy. \citet{antil2021fractional} extended the standard diffusion maps algorithm to the fractional setting that involve the use of a non-local kernel. Fractional Laplacian regularization was applied in \citet{antil2020bilevel} to study tomographic reconstruction. \citet{dunlop2020large} studied large graph limits of semi-supervised learning problems via powers of graph Laplacian, including fractional Laplacians and provided their consistency guarantee in terms of $\Gamma$-convergence. Building on this, semi-supervised learning with finite labels was explored in \citet{weihs2023consistency} via minimizing the fractional Sobolev seminorm/energy and consistency of such approach was provided through showing the $\Gamma$-convergence.

There is an extensive literature on nonparametric statistics on estimating piecewise constant or polynomial functions; see, for example,~\citet{chaudhuri1994piecewise,donoho1997cart,scott2006minimax, tibshirani2014adaptive} and references therein for a sampling of such work. While most of this work focuses on the one or two dimensional setting, \citet{chatterjee2021adaptive} recently considered the multivariate setting and established adaptive rates. Many of these contributions either consider fixed (lattice-based) designs or axis-aligned partitions. Under
appropriate boundary conditions on the shape of the partition cells (not necessarily axis-aligned), piecewise constant or polynomial functions belong to fractional Sobolev spaces. Furthermore, rates of estimation in the case of H\"{o}lder and Lipschitz functions are well-studied~\citep{gyorfi2002distribution,tsybakov2008ntroduction}. In particular, H\"{o}lder functions on bounded domains (which is typically needed in statistical estimation contexts) belong to fractional Sobolev spaces. The inclusion in the other direction is more complicated, and we refer to~\citet{rybalko2023h} for the state-of-the art results in one-dimension. Estimating functions with bounded variation, in both one and multidimensional settings is also considered in the literature; see, for example,~\citet{mammen1997locally,koenker2004penalized,sadhanala2016total,hutter2016optimal,sadhanala2017higher} and references therein for some representative works. In particular, a recent work by~\citet{hu2022voronoigram} considered the random design setting in the multivariate case and established minimax rates. More technical details regarding the relationship between function spaces of bounded variation and fractional Sobolev spaces are provided in Section~\ref{sec:relationship}.

\section{Preliminaries and Methodology}
\subsection{Laplacian matrices based on $\epsilon$-neighborhood graphs}\label{graph}
For i.i.d data $X_{1},\ldots,X_{n}$ from a distribution $G$ suported on $\mcal{X}\subseteq \mbb{R}^{d}$ with the density $g$, the $\epsilon$-neighborhood graph is defined by setting the vertex set as $\{X_{1},\ldots,X_{n}\}$ and the adjacency matrix with weights:
\begin{align}\label{def:weights}
    w_{i,j}^{\epsilon}:=\eta\left(\frac{\|X_i-X_j\|}{\epsilon}\right)\mbf{1}_{\|X_i-X_j\|\le \epsilon},\quad i,j=1,\ldots,n,
\end{align}
where $\| \cdot\|$ denotes the standard Euclidean norm. Here $\eta \ge 0$ is a non-increasing kernel function and $\epsilon$ is the bandwidth parameter. 

The adjacency matrix is then $W=(w_{i,j}^{\epsilon})_{i,j=1,\ldots,n}$ and the degree matrix $D=(d_{ij})_{i,j=1,\ldots,n}$ is then given by a diagonal matrix with the $i$-th diagonal element as $d_{i}:=\sum_{j=1}^{n}w_{i,j}^{\epsilon}$ for $i=1,\ldots,n$. The associated (unnormalized) graph Laplacian is a matrix on the $\epsilon$-graph defined as:
\begin{equation}\label{def:unnormal}
        L_{n,\epsilon}:=\frac{1}{n\epsilon^{d+2}}(D-W),
\end{equation}
where $1/(n\epsilon^{d+2})$ is a scaling factor to ensure a stable limit. For $u\in \mbb{R}^{n}$, the $i$-th coordinate of the vector $L_{n,\epsilon}u$ is given by
\begin{align}\label{def:weightedL}
    (L_{n,\epsilon}u)_{i}=\frac{1}{n\epsilon^{d+2}}\sum_{j=1}^{n}w_{i,j}^{\epsilon}\left(u_i-u_j\right).
\end{align}
It is well known that the (unnormalized) graph Laplacian \eqref{def:unnormal} is self-adjoint with respect to the Euclidean inner product $\langle \cdot,\cdot\rangle$. We denote by the scaled Euclidean inner product $\langle \cdot,\cdot\rangle_{n}:=n^{-1}\langle \cdot,\cdot\rangle$ and write its corresponding scaled norm as $\|\cdot\|_{n}$.

While we restrict our attention here to the unnormalized graph Laplacian, other forms of the graph Laplacians are also widely used in machine learning tasks; some examples include the normalized Laplacian, the random walk Laplacian and a larger family of the weighted graph Laplacians \citep{hoffmann2022spectral,shi2023adaptive} (which includes the normalized Laplacian and random walk Laplacian as special cases). It is possible to extend the procedure proposed in this paper to the class of weighted graph Laplacians, as done in \citet{shi2023adaptive}. We leave a detailed analysis of the merits of such an extension for future work. 

\subsection{Principal Component Regression via Fractional
Laplacian-Eigenmap}\label{modelandalg}
The eigenmap was first proposed in \citet{belkin2003laplacian} to deal with nonlinear dimensionality reduction and data representation. Recently, \citet{green2021minimax,shi2023adaptive} established minimax optimal rates of nonparametric regression via eigenmap on the (weighted) Laplacian. Here, we propose the following principal components regression with the fractional Laplacian eigenmaps (PCR-FLE) algorithm based on the fractional Laplacian matrix $L_{n,\epsilon}^{s}$ for $0<s<1$:
\begin{itemize}
    \item[(1)] For a given parameter $\epsilon>0$ and a kernel function $\eta$, construct the $\epsilon$-graph according to Section \ref{graph}.
    \item[(2)] Compute the fractional Laplacian matrix $L_{n,\epsilon}^{s}$ based on \eqref{def:unnormal} via its eigen-decomposition $L_{n,\epsilon}^{s}=\sum_{i=1}^{n}\lambda_{i}^{s}v_{i}v_{i}^{T}$ , where $(\lambda_i,v_i)$ are the eigenpairs with eigenvalues $0=\lambda_1 \le \ldots\le \lambda_{n}$ in an ascending order and eigenvectors normalized to satisfy $\|v_{i}\|_{n}=1$, for $i=1,\ldots,n$.
    \item[(3)] Project the response vector $Y=(Y_1,\ldots,Y_{n})^{T}$ onto the space spanned by the first $K$ eigenvectors, i.e., denote by $V_{K}\in \mbb{R}^{n\times K}$ the matrix with $j$-th column as $V_{K,j}=v_{j}$ for $j=1,\ldots,K$ and define
    \begin{align}\label{eq:estimator}
        \hat{f}:=V_{K}V_{K}^{T}Y,
    \end{align}
    as the estimator.
\end{itemize}
We remark here that the entries of the vector $\hat f$ are the in-sample values of the estimator of the regression function $f$. Intuitively speaking, PCR-FLE algorithm can be regarded as a PCR variant by substituting the sample covariance matrix with the fractional Laplacian matrix $L_{n,\epsilon}^{s}$, for $0<s<1$. The spotlight of PCR-FLE, however, lies in its capacity to learn nonsmooth functions by the fractional Laplacian compared to the sample covariance matrix or the (weighted) Laplacian. 

\subsection{Fractional Laplacian operator and fractional Sobolev spaces}\label{sec:frationalsobolev}
In this section, we introduce the function space that we consider for our analysis, the fractional Sobolev space, which includes many nonsmooth functions that are of interest in practice.

\begin{definition}\label{def:fractionalsobolev}
For any $0<s<1$, the $L_2$-fractional Sobolev space $H^{s}(\mcal{X})$ is defined as:
\begin{align*}
    \Big\{u\in L^{2}(\mcal{X}):\int_{\mcal{X}\times\mcal{X}}\frac{|u(x)-u(y)|^{2}}{\|x-y\|^{d+2s}}dxdy<\infty\Big\},
\end{align*}
where $L^{2}(\mcal{X}):=\left\{u:\int_{\mcal{X}}u^{2}(x)dx<\infty \right\}$. 
\end{definition}
Consequently, the fractional Sobolev space is an intermediary space between $L^{2}(\mcal{X})$ and the first-order Sobolev space $H^{1}(\mcal{X})$ (consisting of differentiable functions) with the fractional Sobolev seminorm:
\begin{align*}
    |u|_{H^{s}(\mcal{X})}:=\left(\int_{\mcal{X}\times\mcal{X}}\frac{|u(x)-u(y)|^{2}}{\|x-y\|^{d+2s}}dxdy\right)^{\frac{1}{2}},
\end{align*}
and the fractional Sobolev norm:
\begin{align*}
\|u\|_{H^{s}(\mcal{X})}:=\left(\int_{\mcal{X}}u^{2}(x)dx+|u|^2_{H^{s}(\mcal{X})}\right)^{\frac{1}{2}}. 
\end{align*}
For $M>0$ and $0<s<1$, the class of all functions $u$ such that $\|u\|_{H^{s}(\mcal{X})}\le M$ is called a fractional Sobolev ball denoted by $H^{s}(\mcal{X};M)$ of radius $M$.

The above definition of the fractional Sobolev space $H^{s}(\mcal{X})$ is also linked with the following spectrally defined fractional Sobolev space:
\begin{align}\label{spectrallynorm}
    \mcal{H}^{s}(\mcal{X}):=\left\{u\in L^{2}(\mcal{X}): \sum_{i=1}^{\infty}\Lambda_{i}^{s}a_{i}^{2}<\infty\right\},
\end{align}
where $a_{i}:=\langle u,\phi_{i} \rangle$, for $i\ge 1$ and $\{(\Lambda_{i},\phi_{i})\}_{i=1}^{\infty}$ are the eigenpairs of the Laplace–Beltrami operator $\mcal{L}$ such that for $i\ge 1$:
\begin{align*}
    \mcal{L}\phi_{i}=\Lambda_i \phi_{i}\quad\text{with}\ \frac{\partial}{\partial \mbf{n}}\phi_i=0,\ \text{on}\ \partial{\mcal{X}},
\end{align*}
where $\mbf{n}$ stands for the outer normal vector and $\mcal{L}u=-\text{div}(\nabla u)$. Note also that the continuum limit of~\eqref{def:unnormal} corresponds to $\mcal{L}$ as long as $g$ is uniform. In this discussion, we stick to the case of uniform $g$ for simplicity and remark that the connection holds for a general class of densities.

It has been emphasized in \citet{dunlop2020large} that $\mcal{H}^{s}(\mcal{X})\hookrightarrow H^{s}(\mcal{X})$\footnote{$\hookrightarrow$ stands for continuous embedding.}. The above representation of the fractional Sobolev space is more related to the spectral series regression (see \citet{rice1984bandwidth,green2021minimax}) and semi-supervised learning for missing labels (see \citet{weihs2023consistency}).

Moreover, the fractional Sobolev space $H^{s}(\mcal{X})$ is naturally related to the fractional Laplacian operator $\mcal{L}^{s}$ for $0<s<1$. Readers are referred to \citet{di2012hitchhikers} for more details. Here, given a function $u$ in the Schwartz space of rapidly decaying  $C_{c}^{\infty}(\mcal{X})$\footnote{the function $u$ is compactly supported on $\mcal{X}$} functions, the fractional Laplacian is defined as 
\begin{align}\label{eq:flap}
    \mcal{L}^{s}u(x)=c_{n,s}\text{P.V.}\int_{\mbb{R}^{d}}\frac{u(x)-u(y)}{\|x-y\|^{d+2s}}dy,
\end{align}
where `P.V.' stands for the Cauchy Principle Value and $c_{n,s}:=s2^{2s}\Gamma((d+2s)/2)/\Gamma(1-s)$. \citet[Proposition 3.6]{di2012hitchhikers} show the following relationship between their norms:
\begin{align*}
    |u|_{H^{s}(\mbb{R}^{d})}^{2}=2c_{n,s}^{-1}\|\mcal{L}^{s}u\|_{L^{2}(\mbb{R}^{d})}^{2}.
\end{align*}

\begin{figure*}
    \centering
    \subfigure{\includegraphics[scale=0.2]{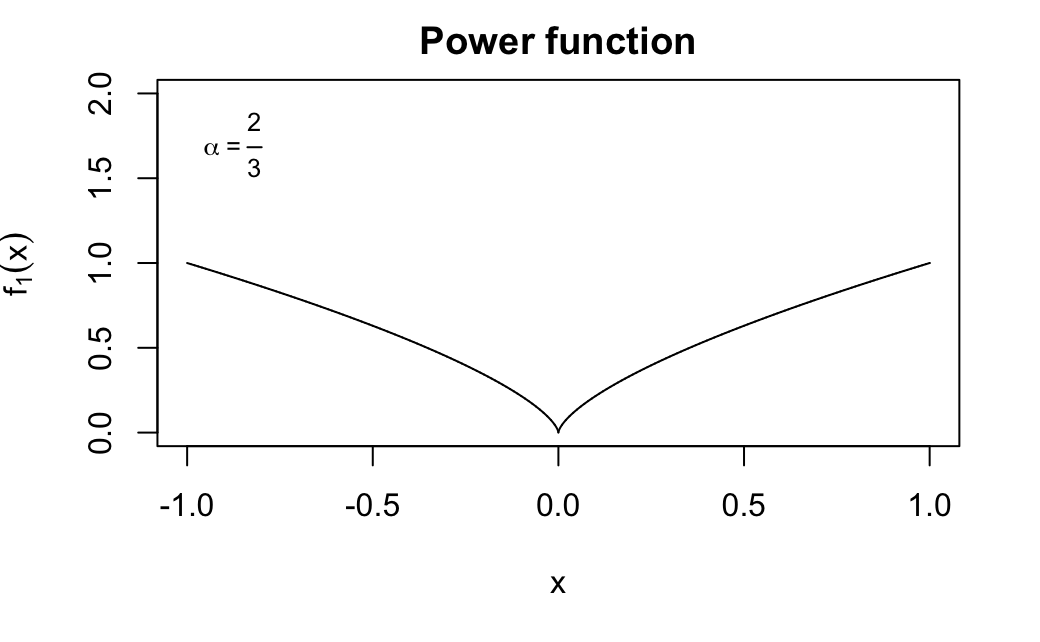}} 
    \subfigure{\includegraphics[scale=0.2]{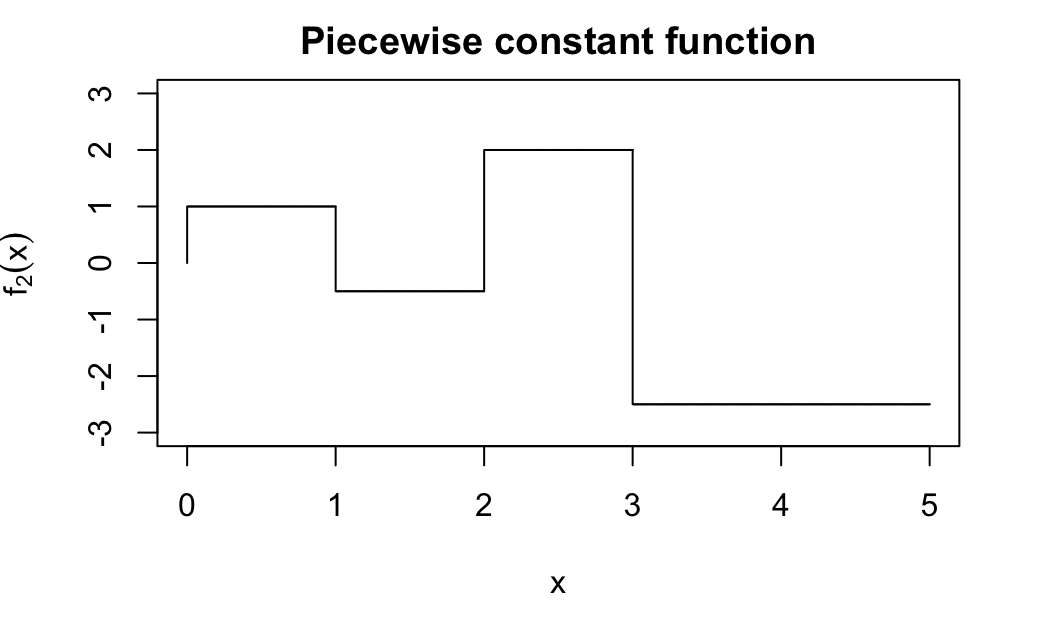}} 
    \subfigure{\includegraphics[scale=0.2]{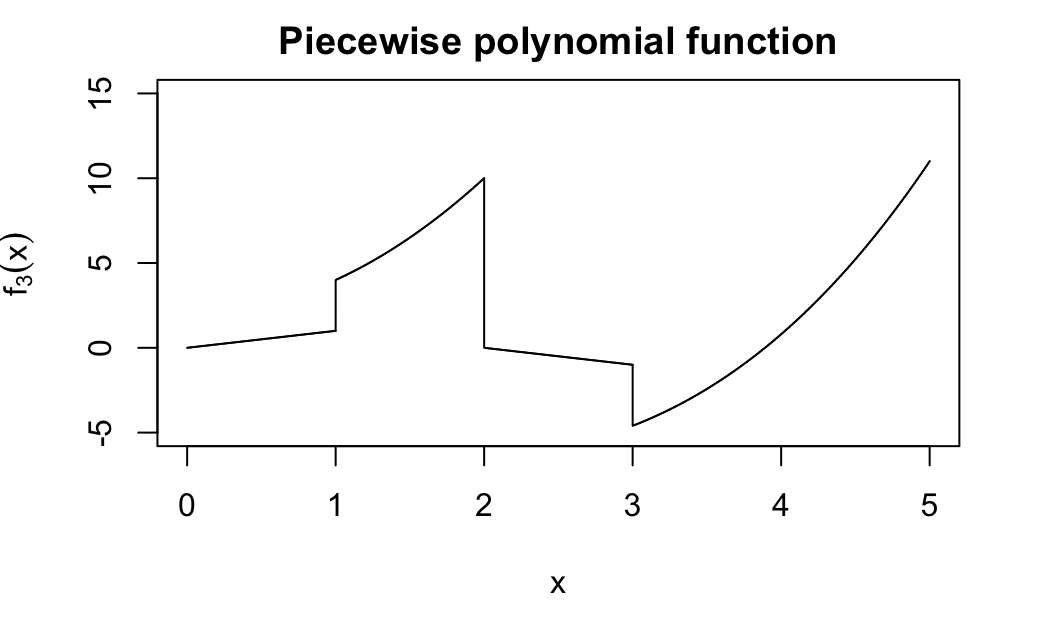}}
    \subfigure{\includegraphics[scale=0.2]{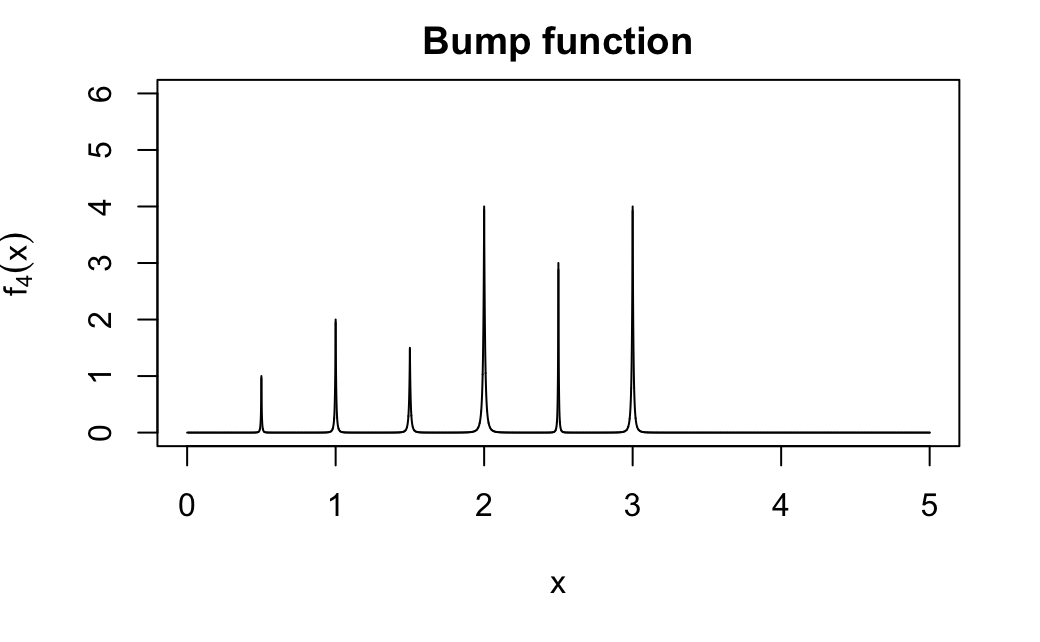}} 
    \caption{Examples of functions that lie in a Fractional Sobolev spaces. The function $f_1$, $f_2$, $f_3$ and $f_4$ are as defined in~\eqref{eq:powerfn}, \eqref{eq:pcfun},~\eqref{eq:ppfun} and~\eqref{eq:bumpfun} respectively.}
    \label{fig:foobar}
\end{figure*}

We end this section by providing some examples of \emph{nonsmooth} functions that are of interest in nonparametric statistics while not covered by the integer-indexed Sobolev space.\\

\textbf{Example 1: Power functions.} It has to be noted that the Sobolev space $H^{s}(\mcal{X})$ for $s\in\mbb{N}_{+}$ not only requires the functions to be $s$-times (weakly) differentiable but the derivatives have to be square-integrable as well. Let's consider the following function:
\begin{align}\label{eq:powerfn}
    f_1(x):=|x|^{\alpha},\quad 0<\alpha<1,
\end{align}
on $(-1,1)$. Obviously, $f_1$ is not (weakly) differentiable at $0$. Furthermore, note that $\int_{0}^{1}x^{2(\alpha-1)}dx=\infty$, for $0<\alpha\le 1/2$. Therefore, it doesn't belong to any integer-indexed Sobolev space when $0<\alpha< 1$. However, let us consider
\begin{align*}
|f_1|_{H^{s}((-1, 1))} = \left( \int_{-1}^{1} \int_{-1}^{1} \frac{||x|^{\alpha} - |y|^{\alpha}|^2 }{|x-y|^{1+2s}} \,dx\,dy \right)^{1/2},
\end{align*}
for $0<\alpha<1$ and $0<s<1$. Note that the singularity is at $x=y$. Furthermore, we have $|x|^{\alpha}-|y|^{\alpha}\sim \alpha |y|^{\alpha-1}(x-y)$ as $x\rightarrow y$ for $y\neq 0$ and $\int_{0}^{1}x^{-p}dx<\infty$ for $p<1$. Hence, the fractional Sobolev seminorm $|f_1|_{H^{\alpha}((-1, 1))}<\infty,$ for $0<s<1$ and $1/2\le \alpha<1$. In summary, the function hence belongs to the fractional Sobolev space $H^{s}((-1,1))$ for $0<s<1$ when $1/2\le \alpha<1$. \\

\textbf{Example 2. Piecewise constant functions.}
Now, consider the following function on $[0,1]$:
\begin{align*}
    f_{\text{pc}}(x)=\left\{
    \begin{aligned}
        &1,\ 0<x\le 1/2,\\
        &0,\ 1/2<x<1.
    \end{aligned}
    \right.
\end{align*}
Clearly, on the same support, $f_{\text{pc}}(x)=f_{\text{pc}}(y)$ for $0<x,y\le 1/2$ or $1/2<x,y<1$. It then suffices to only consider the following integral:
\begin{align*}
    \int_{0}^{\frac{1}{2}} \int_{\frac{1}{2}}^{1} \frac{ 1}{|x-y|^{1+2s}} \,dx\,dy,
\end{align*}
for $0<s<1$. As $\int_{0}^{1}x^{-p}dx<\infty$ for $p<1$, it is finite for $0<s<1/2$. Then, the fractional Sobolev seminorm is finite for $0<s<1/2$, which implies $f_{\text{pc}}(x)$ belongs to the fractional Sobolev space $H^{s}([0,1])$ for $0<s<1/2$.

A generalization of the above example is the piecewise constant functions or blocks (see \cite{donoho1994ideal} for more details) denoted by $f_2(x)$, where the function is constant on multiple intervals that form a partition of $\mcal{X}$. For instance,
\begin{align}\label{eq:pcfun}
    f_2(x)=\left\{
    \begin{aligned}
        &1,\quad 0<x\le 1\\
        &0.5,\quad 1<x\le 2\\
        &2,\quad 2<x\le 3\\
        &-2.5,\quad 3<x< 5.\\
    \end{aligned}
    \right.
\end{align}
The piecewise constant functions/blocks belong to the fractional Sobolev space $H^{s}(\mcal{X})$ for $0<s<1/2$.\\

\textbf{Example 3. Piecewise polynomial functions.}
The piecewise polynomial functions extend the blocks above by putting a polynomial function with degree at most $r\ge 0$ on each interval partition of the support $\mcal{X}$. Furthermore, since we are considering nonsmooth functions, discontinuities at the boundaries of each interval are allowed here. According to the boundedness of the fractional Sobolev seminorms of the power functions $f_1(x)$ and the blocks $f_2(x)$, the piecewise polynomial functions belong to the fractional Sobolev space $H^{s}(\mcal{X})$ for $0<s<1/2$ and all $r\in\mbb{N}$ when $\mcal{X}$ is considered to be an open, connected and bounded subset of $\mbb{R}$ (see Section \ref{sec:assump}). For example, 
\begin{align}\label{eq:ppfun}
    f_3(x)=\left\{
    \begin{aligned}
        &x,\quad 0<x\le 1\\
        &2x^2+2,\quad 1<x\le 2\\
        &-x+2,\quad 2<x\le 3\\
        &0.2x^3-2x-4,\quad 3<x< 5.\\
    \end{aligned}
    \right.
\end{align}

In general dimension $d\ge 1$, the above arguments can be generalized to imply that the piecewise polynomial functions (including the piecewise constant functions) belong to the fractional Sobolev space when each partition of the support $\mcal{X}$ is connected and bounded and the boundary of each partition is a lower dimensional space, which allows non-axis aligned partitions.

\smallskip

\textbf{Example 4: Bumps functions} The (multiple) bumps are functions that decay fast from each peak of the bumps. In signal processing, the bumps with polynomial decay or exponential decay are commonly considered. For example,
\begin{align}\label{eq:bumpfun}
    f_{4}(x):=\sum_{j=1}^{J}h_{j}K\left(\frac{t-t_j}{w_j}\right),
\end{align}
where $K(|t|):=(1+|t|)^{-4}$, and $\{t_j,h_j,w_j\}_{j=1}^{J}$ are parameters that $t_{j}$ are the locations of each peak and $h_{j}$ are the peak values. Due to the polynomial/exponential decay, the bumps $f_4(x)$ also belong to the fractional Sobolev space $H^{s}(\mcal{X})$ for $0<s<1/2$ when $\mcal{X}$ is considered to be an open, connected and bounded subset of $\mbb{R}$ (see Section \ref{sec:assump}).


We emphasize that the the nonsmooth functions presented above (including their extensions in $\mbb{R}^{d}$) are representative of spatially variable functions arising in imaging, spectroscopy and other signal processing applications that are of considerable practical importance. We refer to \citet{donoho1994ideal,boudraa2004emd,liu2016variational,sardy2000block,sardy2001robust} for additional exposition of the examples.

\subsection{Relationship to other function classes}\label{sec:relationship}
First note that according to the definition of the fractional Sobolev space (i.e.,  Definition \ref{def:fractionalsobolev}), bounded H\"{o}lder functions of order $\alpha>0$ on bounded domains belong to the fractional Sobolev space for $0<s<(\alpha\wedge 1)$.

\citet{hu2022voronoigram} considered nonparametric estimation with general measure-based bounded total variation class, with finite $L_{\infty}$ norms. For this class, they developed minimax estimators. In particular, the total variation norm was with respect to the $L_1$ norm of the weak derivatives. Moreover, \citet{fang2021multivariate} investigated a similar class: the bounded variation in the sense of Hardy-Krause. The fractional Sobolev space that we focus on in this work are set to be a subspace of $L_{2}$ space, the corresponding norms are function-value based and do not require weak derivatives to exist. Furthermore, the class of fractional Sobolev spaces can be considered with respect any $L_{p}$ space for $p \ge 1$ (see \cite{di2012hitchhikers} for the definition). In general, both bounded variation functional space and fractional Sobolev space can characterise nonsmooth functions. When specializing in the indicator functions, the bounded variation functional space only contains such functions with locally finite perimeter for the support. The exact inclusion relationships between the spaces of bounded variation functions (with respect to the weak derivatives) and $L_1$ or $L_2$ based fractional Sobolev spaces are not well-explored in the literature, to the best of our knowledge.

\citet{rockova2021ideal} studied Bayesian estimators when the truth lies in the set of locally H\"{o}lder functions with finite $L_{\infty}$ norm. When considering the bounded support $\mcal{X}$, locally H\"{o}lder functions belong to the fractional Sobolev space. However, in general, the H\"{o}lder functions include functions that may not even be in $L_{1}$ or $L_{2}$. \citet{imaizumi2019deep} applied deep neural networks to learn a class of nonsmooth functions that are piecewise H\"{o}lder. Similar to locally H\"{o}lder functions, when considering the bounded support $\mcal{X}$, bounded piecewise H\"{o}lder functions belong to the fractional Sobolev space while in general, the former one allows functions not necessarily in $L_{1}$ or $L_{2}$.

Intuitively speaking, when characterising nonsmooth functions, compared to the aforementioned functional spaces, the fractional Sobolev space tends to allow `worse' local non-smoothness while requiring `better' global smoothness (in $L_{1}$ or $L_{2}$).

\section{Theoretical Results}
Before stating our assumptions and results, we introduce some conventions. For two real-valued quantities, $A,B$, the notation $A \lesssim B$ means that there exists a constant $C > 0$ not depending on $f$, $M$ or $n$ such that $A\le CB$ and $A\asymp B$ stands for $A\lesssim B$ and $B\lesssim A$. Also, applying the scaled Euclidean norm $\|\cdot\|_n$ or the corresponding scaled dot-product $\langle\cdot,\cdot\rangle_n$ to a function $f$, is to be understood as applying it to the vector in-sample evaluations $(f(X_1),\ldots,f(X_n))$ of the function.

\subsection{Assumptions}\label{sec:assump}
We list the following major assumptions needed for the sampling distribution/density and the kernel $\eta$. 

\begin{itemize}
    \item[\namedlabel{a1}{(A1)}] 
    The distribution $G$ is supported on $\mcal{X}$, which is an open, connected, and bounded subset of $\mbb{R}^{d}$ with Lipschitz boundary.
    \item[\namedlabel{a2}{(A2)}] 
    The distribution $G$ has a density $g$ on $\mcal{X}$ such that
    \begin{align*}
        0<g_{\min}\le g(x)\le g_{\max}<\infty,\ \text{for all}\ x\in \mcal{X},
    \end{align*}
    for some $g_{min},g_{\max}>0$. Additionally, $g$ is Lipschitz on $\mcal{X}$ with Lipschitz constant $L_{g}>0$.
    \item[\namedlabel{a3}{(A3)}] 
    The kernel $\eta$ is a non-negative, monotonically non-decreasing function supported on the interval $[0,1]$ and its restriction on $[0,1]$ is Lipschitz and for convenience, we assume $\eta(1/2)>0$ and define
    \begin{align*}
    \sigma_{0}:=\int_{\mbb{R}^{m}}\eta(\|x\|)dx,\ \sigma_{1}:=\frac{1}{d}\int_{\mbb{R}^{m}}\|y\|^{2}\eta(\|y\|)dy.
    \end{align*}
    Without loss of generality, we will assume $\sigma_0=1$ from now on.
\end{itemize}

Assumptions $(A1)$ and $(A2)$ are mild and standard assumptions on the density function in the field of graph Laplacians, which are also made in~\citet{green2021minimax,shi2023adaptive,garcia2020error}. Assumption $(A3)$ is a standard normalization condition made on the smoothing kernel; see \citet{garcia2020error} for more details. The requirement that $\eta$ is compactly supported is purely due to our proof technique. While it is in principle possible to generalize it for non-compact kernels as long as the tails decay relatively fast including the Gaussian kernel, that would require obtaining error bounds on extra terms on the tail, which is beyond the scope of this paper.

\subsection{Estimation error of PCR-FLE algorithm}

\begin{theorem}\label{mainthm}
    Let Assumptions $(A1)$-$(A3)$ hold, and further assume $f\in H^{s}(\mcal{X};M)$ for $0<s<1$ and $M>0$. Suppose there exist constants $c_0,C_0>0$ such that 
    \begin{align*}
         c_0\left(\frac{\log n}{n}\right)^{\frac{1}{d}}&\le \epsilon \le C_0  K^{-\frac{1}{d}},
    \end{align*}
    with 
    \begin{align}\label{conK}
        K=\min\left\{\lfloor (M^{2}n)^{\frac{d}{2s+d}}\rfloor \vee 1,n\right\}.
    \end{align}
    Then, there exist constants $c,C>0$ not depending on $f,M$ or $n$ such that for $n$ large enough, the estimator $\hat f$ defined in~\eqref{eq:estimator} satisfies:
    \begin{align*}
        \|\hat{f}-f\|_{n}^{2}\le C\big\{\big(M^2(M^2n)^{-\frac{2s}{2s+d}}\wedge 1\big) \vee {n^{-1}}\big\},
    \end{align*}
with probability at least $1-Cn^2e^{-cn\epsilon^{d+4}}-Cne^{-cn}-Cne^{-cn\epsilon^{d}}-e^{-K}$.
\end{theorem}

\begin{remark}
    Theorems \ref{mainthm}
    implies that the PCR-FLE algorithm achieves an upper bound of rates $n^{-2s/(2s+d)}$ with respect to the fractional Sobolev spaces $H^{s}(\mcal{X})$ for $0<s<1$ with high probability, provided that $n^{-1/2}\lesssim M\lesssim n^{s/d}$. Recall that the minimax optimal rates for the integer-valued Sobolev space $H^{s}(\mcal{X})$ ($s\in\mbb{N}_{+}$) is given by $M^{2}(M^{2}n)^{-\frac{2s}{2s+d}}$ in \citet{gyorfi2002distribution,wasserman2006all,tsybakov2008ntroduction}. While allowing $s\rightarrow 1^{-}$, it is consistent with the above rates for the first-order Sobolev space $H^{1}(\mcal{X})$.
\end{remark}

\begin{remark}
Under a fixed-design setup (i.e., a regular lattice/grid), \citet{chatterjee2021adaptive} considered optimal regression tree (ORT) and showed that the finite sample risk of ORT is always bounded by $\frac{C(r)k\log N}{N}$ for some constant $C(r)>0$ and $N=c^{d}$ for some grid size $c>0$ when the regression function is piecewise polynomial of degree $r$ on some reasonably regular axis-aligned rectangular partition of the domain with at most $k$ rectangles. While such piecewise polynomial regression function belongs to the fractional Sobolev space, our bound in Theorem \ref{mainthm} is valid for a larger family of nonsmooth functions and allows random design set-up.
\end{remark}

\begin{remark}
    A phase transition in the fractional Sobolev space $H^{s}(\mcal{X})$ was discussed in \citet[Lemma 4]{dunlop2020large} that the regularity of the fractional Sobolev space depends on $s<d/2$ or $s>d/2$ (when $s<d/2$, $H^{s}(\mcal{X})$ cannot even embed continuously into the space of continuous functions $C^{0}(\mcal{X})$). However, it should be noted that Theorem \ref{mainthm} does not require the condition $s>d/2$ regardless of the phase transition.
\end{remark}

\begin{remark}\label{resoneps}
    The lower bound for $\epsilon$ makes sure that with this smallest radius, the resulting graph will still be connected with high probability and the upper bound for $\epsilon$ ensures the eigenvalue of the graph Laplacian to be of the same order as its  continuum version, the eigenvalue of the Laplacian operator (Weyl's law). The condition on $K$ is set to trade-off bias and variance. 
\end{remark}

\begin{remark}
    For computing the eigen-decomposition, we can leverage efficient sparse eigen-decomposition algorithms (e.g., Lanczos or randomized SVD) that scale nearly linearly in $n$ for sparse graphs. This is so, because we only require computing the top-$K$ eigenvectors of the graph Laplacian, where $k\ll n$ with $K=O(n^{\frac{d}{2s+d}})$, where $d/(2s+d)$ does not explode in higher dimensions, and the $\epsilon$-neighborhood graph constructed is sparse by design. Moreover, in the context of other similar problems like graph-based semi-supervised learning, conjugate-gradient based methods have been proven useful in obtaining speedups for large but sparse graphs. See \cite{sharma2023efficiently} for details.
\end{remark}

\begin{remark}
    \citet{antil2020bilevel} considered nonparametric regression via fractional Laplacian regularization. However, no convergence rates of any kind were investigated there. On the other hand, it has been discussed and emphasized in \citet{green2021minimaxsmoothing,green2021minimax} that Laplacian regularization usually achieves worse minimax rates of convergence especially in high-dimensional space $\mbb{R}^{d}$ compared to Laplacian eigenmaps. 
\end{remark}

\begin{remark}
    Although our current theoretical analysis assumes homoscedastic Gaussian noise, the PCR-FLE algorithm could be extended to heteroscedastic noise models with minimal changes to the proof. This would require replacing the standard chi-squared concentration with concentration inequalities for quadratic forms with non-constant variance (e.g., using Bernstein-type inequalities). We expect the upper bound to remain of the same order under mild regularity assumptions on the noise variance function.
\end{remark}

\subsection{Lower bound and minimax optimality}
For an estimator $\hat{f}_n$, define its integrated mean-squared estimation error as $\mbb{E}\|\hat{f}_n-f\|^2:=\int_{\mcal{X}}(\hat{f}_n(x)-f(x))^2g(x)dx$. The following theorem establishes a minimax lower bound in the integrated mean-squred estimation error for estimating functions in $H^s(\mathcal{X},M)$.

\begin{theorem}\label{thm:lowerbound}
    Suppose $f\in H^s(\mcal{X};M)$ for $0<s<1$ and the density $g$ is uniform on $\mcal{X}$. Then, there exists a constant $C_1>0$ independent of $M,n$ such that $n^{-\frac{2s}{2s+d}}$ is a lower minimax rate of convergence. In particular,
    \begin{align*}
        \underset{n\rightarrow \infty}{\lim\inf} \inf_{\hat{f}_n}\sup_{f\in H^s(\mcal{X},M)}~\frac{\mbb{E}\|\hat{f}_n-f\|^2}{M^{\frac{2d}{2s+d}}n^{-\frac{2s}{2s+d}}}\ge C_1>0.
    \end{align*}
\end{theorem}

The above result allows random design set-up compared to the existing works such as \citet{chatterjee2021adaptive}. The proof involves generalizing the arguments in \citet[Proof of Theorem 3.2]{gyorfi2002distribution} to handle the non-smoothness in fractional Sobolev spaces.

\begin{remark}
    Combing with Theorem \ref{mainthm}, Theorem \ref{thm:lowerbound} etablishes the minimax optimality of the proposed PCR-FLE algorithm. That is, no other estimator can perform better than the PCR-FLE method, up to constant factors. 
\end{remark}

\section{Numerical Experiments}\label{eq:sim}

In this section, we empirically demonstrate the performance of the PCR-FLE algorithm in Section \ref{modelandalg} for learning nonsmooth regression functions. Particularly, in our experiments, we stick to considering those functions that are of practical importance as introduced in Section \ref{sec:frationalsobolev}. For simplicity, we set the design distribution $G$ as the uniform distribution and examine the piecewise polynomial (including piecewise constant/the blocks) functions as the true regression function. For the construction of graph Laplacian, we pick a truncated Gaussian kernel. Unless otherwise stated, all tuning parameters are set as the optimal values according to grid search and each experiment is averaged over 200 repetitions. 

\begin{figure*}[t]
	 \centering
	\subfigure{\includegraphics[scale=0.5]{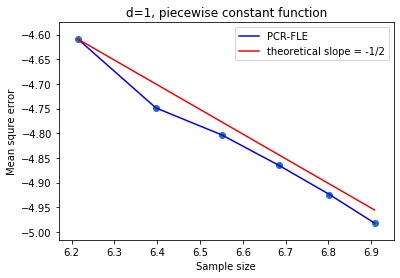}}
 \subfigure{\includegraphics[scale=0.5]{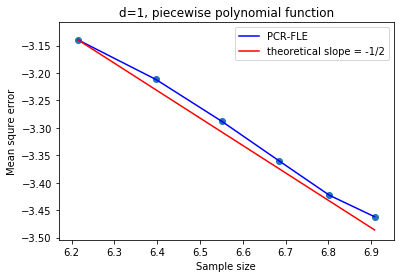}}
	\caption{In-sample mean squared error of PCR-FLE as a function of the sample size $n$. Each subplot is on the $\log$-$\log$ scale. The blue line presents the empirical error by PCR-FLE. The red line shows the theoretical upper bound provided by Theorem \ref{mainthm} (in slope only and the intercept is set to match the observed error).}
 \label{3}
\end{figure*}

\begin{figure*}[t]
	 \centering
    \subfigure{\includegraphics[scale=0.5]{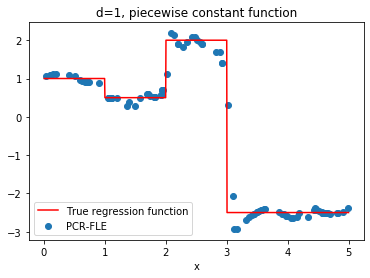}}
    \subfigure{\includegraphics[scale=0.5]{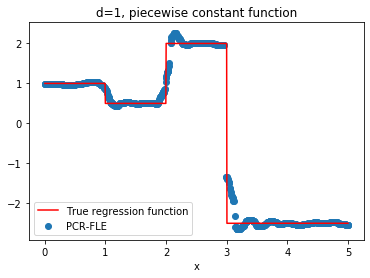}}
	\caption{The red line shows the true regression function. The blue line shows the average of the regression fit by PCR-FLE estimation condition on a generation of uniform sample on $[0,5]$: $n=100$ (top) and $n=1000$ (bottom). See Figure \ref{2a} (in Section~\ref{sec:addexp}) for a similar visualization for piecewise polynomial functions.}
 \label{2}
\end{figure*}

\textbf{Estimation.} We now consider the mean squared error of the PCR-FLE estimation on the nonsmooth functions: the piecewise constant function and the piecewise polynomial function ($f_2(x)$ and $f_{3}(x)$ respectively in Figure \ref{fig:foobar}). Due to relatively rapid rate of convergence, the sample size $n$ is set to vary from 500 to 1000. 

In Figure \ref{3} ($\log$-$\log$ scale), we show the in-sample mean squared errors of both estimators as a function of the sample size $n$. We see that both estimators have mean squared error converging to 0 roughly at our theoretical rate in Theorem \ref{mainthm} while this provides a high probability upper bound. In Figure \ref{2} and Figure \ref{2a} (in Section~\ref{sec:addexp}), we present the fitted regression function by PCR-FLE, visually. 

\section{Discussion}
We proposed and analyzed the PCR-FLE algorithm for performing nonparametric regression when the true function lies in the $L_2$-fractional Sobolev space, $H^s(\mathcal{X},M)$. The approach is computational efficient and it involves computing the top-$K$ eigenvalues and eigenvectors of  size $n\times n$ graph Laplacian matrix. Under a random design setting, we established minimax rates of convergence of order $n^{-\frac{2s}{2s+d}}$, where $n$ is the number of observations. There are several avenues for future works:
\begin{itemize}[noitemsep]
    \item Our current results require knowledge of $s$ and $M$ in setting the bandwidth parameter $\epsilon$ and the number of eigenvalues $K$. It is interesting to develop estimators that are adaptive to the choice of $s$ and $M$, by extending the recent results in~\citet{shi2023adaptive}.
    \item It is interesting to go beyond $L_2$-fractional Sobolev spaces and consider $L_1$-fractional Sobolev spaces which allow for richer class of nonsmooth true functions.
\end{itemize} 

\bibliography{uai2025-template}

\newpage

\onecolumn

\title{Minimax Optimal Nonsmooth Nonparametric Regression  via Fractional Laplacian Eigenmaps\\(Supplementary Material)}
\maketitle

\appendix
\section{Additional Experimental Results}\label{sec:addexp}

\begin{figure}[!h]
	 \centering

 \subfigure{\includegraphics[scale=0.5]{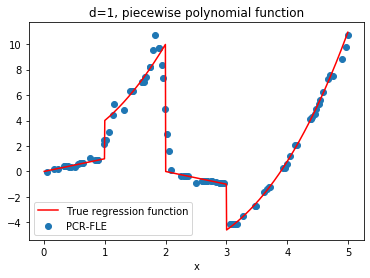}}
   \subfigure{\includegraphics[scale=0.5]{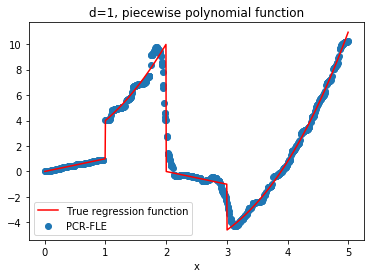}}
	\caption{The red line shows the true regression function. The blue line shows the expectation of the regression function estimated by PCR-FLE: $n=100$ (left) and $n=1000$ (right).}
 \label{2a}
\end{figure}

\newpage
\section{Pseudocode of the PCR-FLE algorithm}
\begin{algorithm}[h!]
\caption{PCR-FLE: Principal Component Regression via Fractional Laplacian Eigenmaps}
\begin{algorithmic}[1]
\REQUIRE Data $\{(X_i, Y_i)\}_{i=1}^n$, bandwidth $\epsilon > 0$, fractional order $s \in (0, 1)$, number of components $K$, kernel function $\eta$
\ENSURE Estimated regression values $\hat{f} \in \mathbb{R}^n$

\STATE \textbf{Construct $\epsilon$-neighborhood graph:}
\FOR{$i = 1$ to $n$}
    \FOR{$j = 1$ to $n$}
        \IF{$\|X_i - X_j\| \leq \epsilon$}
            \STATE $w_{ij} \leftarrow \eta\left(\|X_i - X_j\| / \varepsilon\right)$
        \ELSE
            \STATE $w_{ij} \leftarrow 0$
        \ENDIF
    \ENDFOR
\ENDFOR
\STATE $W \leftarrow (w_{ij})$, $D_{ii} \leftarrow \sum_j w_{ij}$, and $L_{n, \epsilon} \leftarrow \frac{1}{n \epsilon^{d+2}}(D - W)$

\STATE \textbf{Eigen-decomposition:}
\STATE Compute eigenpairs $(\lambda_i, v_i)$ of $L_{n, \epsilon}$: $\lambda_1 \leq \cdots \leq \lambda_K$ with the corresponding eigenvectors $v_1,\ldots,v_K$.

\STATE \textbf{Project response onto top-$K$ eigenvectors:}
\STATE Form $V_K = [v_1, \ldots, v_K] \in \mathbb{R}^{n \times K}$
\STATE Compute projection: $\hat{f} = V_K V_K^\top Y$

\RETURN $\hat{f}$
\end{algorithmic}
\end{algorithm}

\newpage
\section{Proof of Theorem \ref{mainthm}}

In this section, we will prove Theorem \ref{mainthm}. To this end, we first present some auxiliary lemmas. In the following, $C$ stands for positive constants that may change from line to line but do not depend on $n$ or $M$.

\begin{lemma}[Weyl’s Law]\label{WeylLaw}
    Suppose Assumptions \ref{a1} and \ref{a2} hold. There exist constants $c,C>0$ such that
    \begin{align*}
        ck^{\frac{2}{d}}\le \Lambda_{k}\le Ck^{\frac{2}{d}},\quad \text{for all}\ k\ge 1.
    \end{align*}
    Here, $\Lambda_{k}$ is the $k$-th eigenvalue of the weighted Laplacian operator $\mcal{L}_{g}$ in the ascending order, where $\mcal{L}_{g}u:=-\frac{1}{2g}\text{div}(g^{2}\nabla u)$.
\end{lemma}
The Weyl's law is a standard result in operator analysis. We refer interested readers to \citet[Lemma 7.10]{dunlop2020large} for a detailed proof.

\begin{lemma}[Lemma 2 in \citet{green2021minimaxsmoothing}]\label{eigenvaluebound}
    There exist constants $C_1,C_2,C_3,C_4,C_5>0$ such that for $n$ large enough and $C_1(\log n/n)^{\frac{1}{d}}\le \epsilon\le C_2$, with probability at least $1-C_3ne^{-C_3n\epsilon^{d}}$, it holds that
    \begin{align*}
        C_4 \left(k^{\frac{2}{d}}\wedge r^{-2}\right)\le \lambda_{k}\le C_5 \left(k^{\frac{2}{d}}\wedge r^{-2}\right),\quad\text{for}\ 2\le k\le n.
    \end{align*}
\end{lemma}
The above result actually is different from similar results established in~\citet{burago2015graph,garcia2020error,calder2022improved}, who establish similar results for manifolds without boundary. We refer to~\citet[Appendix D]{green2021minimax} for a more detailed discussion. 

We are now in the position to prove the main Theorem \ref{mainthm}.
\begin{proof}[Proof of Theorem \ref{mainthm}]
By Cauchy-Schwarz inequality, we have:
\begin{align*}
    \|\hat{f}-f\|_{n}^2\le 2(\|\mbb{E}\hat{f}-f\|_{n}^2+\|\hat{f}-\mbb{E}\hat{f}\|_{n}^2).
\end{align*}
Then, according to PCR-FLE algorithm in Section \ref{modelandalg}, we obtain
\begin{align}\label{decomp:bias}
    \|\mbb{E}\hat{f}-f\|_{n}^{2}=\sum_{k=K+1}^{n}\langle v_k,f \rangle_{n}^2\le \frac{\langle L_{n,\epsilon}^{s}f,f\rangle_{n}}{\lambda_{K+1}^{s}},
\end{align}
and
\begin{align*}
    \|\hat{f}-\mbb{E}\hat{f}\|_{n}^2=\sum_{k=1}^{K}\langle v_k,\varepsilon\rangle_{n}^{2},
\end{align*}
where $\varepsilon:=(\varepsilon_1,\ldots,\varepsilon_n)^{T}$.

Now, note that if $K=0$, $\hat{f}=0$ then $\|\hat{f}-\mbb{E}\hat{f}\|_{n}^2=0$. We then focus on the case when $K>1$. Since $\langle v_k,\varepsilon\rangle_{n}$ is normally distributed with $0$ mean and variance:
\begin{align}\label{var1}
    \Var \langle v_k,\varepsilon\rangle_{n}=\frac{1}{n^2} \Var \langle v_k,\varepsilon\rangle=\frac{1}{n},
\end{align} 
as $\langle v_k,v_k\rangle=n$. Then, we obtain:
\begin{align*}
    \|\hat{f}-\mbb{E}\hat{f}\|_{n}^2=\frac{1}{n}\sum_{k=1}^{K}(\sqrt{n}\langle v_k,\varepsilon\rangle_{n})^{2}\overset{d}{=}\frac{1}{n}\sum_{k=1}^{K}\mcal{Z}_{k}^{2},
\end{align*}
where $\{\mcal{Z}_{k}\}_{k=1}^{K}$ are i.i.d. standard normal by the orthonormality of the eigenvectors.

According to an exponential concentration inequality for chi-square distributions from \citet{laurent2000adaptive}, we have
\begin{align}\label{decomp:variance}
    \mbb{P}\left(\|\hat{f}-\mbb{E}\hat{f}\|_{n}^2\ge \frac{K}{n}+2\frac{\sqrt{K}}{n}\sqrt{t}+2\frac{t}{n}\right)\le e^{-t}.
\end{align}
With \eqref{decomp:bias} and \eqref{decomp:variance}, it yields that
\begin{align}\label{biasvariancedecomp}
    \|\hat{f}-f\|_{n}^{2}\le \frac{\langle L_{n,\epsilon}^{s}f,f\rangle_{n}}{\lambda_{K+1}^{s}}+\frac{K}{n},
\end{align}
with probability at least $1-e^{-K}$ if $1\le K\le n$. Moreover, when $K=0$, \eqref{biasvariancedecomp} holds immediately.

Now, it remains to bound the empirical fractional Sobolev seminorm $\langle L_{n,\epsilon}^{s}f,f\rangle_{n}$ and the (power of) graph Laplacian eigenvalue $\lambda_{K+1}$ ($\lambda_{K+1}^{s}$) for $0<s<1$.

Now, we first focus on the empirical fractional Sobolev seminorm $\langle L_{n,\epsilon}^{s}f,f\rangle_{n}$ for $0<s<1$. Note that \citet{weihs2023consistency,dunlop2020large,trillos2018variational} showed the $\Gamma$-convergence of the above empirical fractional Sobolev seminorm to its continuum in \eqref{spectrallynorm}. However, $\Gamma$-convergence does not fully suffice for our purpose. Instead, we will adapt the proof procedures applied in \citet{calder2022improved,green2021minimaxsmoothing,green2021minimax} for our situation, as we describe next.

First note that according to the eigendecomposition of $L_{n,\epsilon}$, we obtain:
\begin{align}\label{discreteseminorm}
    \langle L_{n,\epsilon}^{s}f,f\rangle_{n}=\sum_{i=1}^{n}\lambda_{i}^s\langle f,v_{i}\rangle_{n}^2.
\end{align}
Now, by Lemma \ref{WeylLaw} and Lemma \ref{eigenvaluebound}, we have that
\begin{align*}
    \lambda_{i}\lesssim \Lambda_{i},
\end{align*}
for $1\le i\le n$ with probability at least $1-Cne^{-cn\epsilon^{d}}$. We now focus on the eigenvectors $\{v_{i}\}_{i=1}^{n}$. For a given eigenvalue $\Lambda>0$ of $\mcal{L}_{g}$, assume $\Lambda=\Lambda_{i+1}=\ldots=\Lambda_{i+k}$ for some $i$ and $k$, where $k$ is multiplicity of $\Lambda$. We then define the eigenvalue gap of $\Lambda$ as:
\begin{align}\label{eigengap}
    \gamma_{\Lambda}:=\frac{1}{2}\left(|\Lambda-\Lambda_{i}|\wedge |\Lambda-\Lambda_{i+k+1}|\right).
\end{align}

Now, according to \citet[Proof of Theorem 6]{green2021minimaxsmoothing}, we can pick $\epsilon$ small enough and constants $A,\theta,\tilde{\delta}>0$ such that
\begin{align*}
    1-A\left(\epsilon\sqrt{\Lambda}+\theta+\tilde{\delta}\right)\ge \frac{1}{2},
\end{align*}
where $\theta$ and $\tilde{\delta}$ are given in \citet[Equation (36), Section D]{green2021minimaxsmoothing} and $A>0$ is defined in \citet[Proof of Theorem 6]{green2021minimaxsmoothing} with $A\ge 2$. Then, an application of \citet[Theorem 6]{green2021minimaxsmoothing} yields that for such $\Lambda=\Lambda_{l}$ ($l$-th eigenvalue of $\mcal{L}_{g}$ in the ascending order), with probability at least $1-Cne^{-cn\theta^{2}\tilde{\delta}^{d}}$,
\begin{align}\label{rougheigenvalueapprox}
    a\lambda_{l}\le \sigma_1\Lambda_{l}\le A\lambda_{l},
\end{align}
where $a$ is given in \citet[Proof of Theorem 6]{green2021minimaxsmoothing} with $a^{-1}\ge 2$. Then, we have with probability at least $1-Cne^{-cn\theta^{2}\tilde{\delta}^{d}}$:
\begin{align}\label{eigenvalueclose1}
    |\lambda_l-\sigma_1 \Lambda_l|\le ((a^{-1}-1)\vee (A-1))\sigma_1\Lambda_{l}\le C\gamma_{\Lambda_{l}}.
\end{align}

Let $S$ be the subspace of $l^{2}$ spanned by the eigenvectors of $L_{n,\epsilon}$ associated to the eigenvalues $\lambda_{i+1},\ldots,\lambda_{i+r}$. In the following, we will establish the bound on the eigenfunctions/eigenvectors of $\Lambda_{j}$ and $\lambda_{j}$ respectively for $j=i+1,\ldots,i+r$. Denote by $P_{S}$ the orthogonal projection (with respect to $\langle\cdot,\cdot\rangle_{n}$) onto $S$ and $P_{S}^{\bot}$ the orthogonal projection onto the orthogonal complement of $S$. Let $h$ be the eigenfunction of $\mcal{L}_{g}$ corresponding to the eigenvalue $\Lambda$, i.e., $\mcal{L}_{g}h=\Lambda h$. Considering restriction of $h$ on $X_1,\ldots,X_{n}$, we have
    \begin{align*}
        P_{S}^{\bot}\mcal{L}_{g}h=\Lambda P_{S}^{\bot}h=\Lambda \sum_{j\neq i+1,\ldots,i+r}\langle h,v_{j}\rangle_{n}v_{j},
    \end{align*}
    where recall that $\{v_{j}\}_{j=1}^{n}$ are the set of the orthonormal basis of eigenvectors of $L_{n,\epsilon}$ with respect to $\lambda_{1},\ldots,\lambda_{n}$. Similarly, we have (again, restrict $h$ on $X_1,\ldots,X_n$):
    \begin{align*}
        P_{S}^{\bot}L_{n,\epsilon}h= \sum_{j\neq i+1,\ldots,i+r}\lambda_{j}\langle h,v_{j}\rangle_{n}v_{j}.
    \end{align*}
    Combining the two results above, we obtain:
    \begin{align*}
        &\quad \min\{|\sigma_1 \Lambda-\lambda_{i}|,|\sigma_1 \Lambda-\lambda_{i+r+1}|\}\|P_{S}^{\bot}h\|_{n}\\&\le \|P_{w,S}^{\bot}(L_{w,n,\epsilon}h-\sigma_1 \mcal{L}_{g}h)\|_{n}\\&\le \|L_{n,\epsilon}h-\sigma_1 \mcal{L}_{g}h\|_{n},
    \end{align*}
    where $\sigma_1$ is defined in Section \ref{sec:assump} (see Assumption (A3)).
    
    On the other hand, according to \eqref{eigengap} and \eqref{eigenvalueclose1}, we have
    \begin{align*}
        \min\{|\sigma_1 \lambda-\lambda_{i}(L_{n,w,\epsilon})|,|\sigma_1 \lambda-\lambda_{i+r+1}(L_{w,n,\epsilon})|\}\ge \sigma_1 C\gamma_{\Lambda}.
    \end{align*}
    Then, we obtain:
    \begin{align}\label{eigenvfclose}
        \|P_{S}^{\bot}h\|_{n}=\|h-P_{S}h\|_{n}\le \frac{1}{\sigma_1 C\gamma_{\Lambda}}\|L_{n,\epsilon}h-\sigma_1 \mcal{L}_{g}h\|_{n}.
    \end{align}
    Now, we divide the above norm $\|\cdot\|_{n}$ into two parts by Cauchy's inequality:
    \begin{align*}
        \|L_{n,\epsilon}h-\sigma_1 \mcal{L}_{g}h\|_{n}\lesssim \|L_{n,\epsilon}h-\sigma_1 \mcal{L}_{g}h\|_{n,\mcal{X}_{\epsilon}}+\|L_{n,\epsilon}h-\sigma_1 \mcal{L}_{g}h\|_{n,\partial  \mcal{X}_{\epsilon}},
    \end{align*}
where we write $\mcal X = \mcal{X}_{\epsilon} \sqcup \partial  \mcal{X}_{\epsilon}$, where for any $x\in \mcal{X}_{t\epsilon}$, $B_{x}(\epsilon)\subset \mcal{X}$ and $\partial_{\epsilon} \mcal{X}$ as its complement within $\mcal{X}$ consisting of points `close' to the boundary. According to \citet[Theorem 3.3]{calder2022improved}, it follows that if $h_1,\ldots,h_k$ is an orthonormal basis for the eigenspace of eigenfunctions of $\mcal{L}_{g}$ with respect to eigenvalue $\Lambda$, then with probability at least $1-2kne^{-Cn\epsilon^{d+4}}$,
\begin{align*}
    \|L_{n,\epsilon}h_{j}-\sigma_1 \mcal{L}_{g}h_{j}\|_{n,\mcal{X}_{\epsilon}}\le C \epsilon,\quad 1\le j\le k.
\end{align*}
On the other hand, near the boundary, by setting $k=1$ and $s=3$ (since all $h_{j}$ at least belongs to $C^{3}(\mcal{X})$) in \citet[Lemma 5]{green2021minimax}, it yields that almost surely,
\begin{align*}
    \|L_{n,\epsilon}h_{j}\|_{n,\partial  \mcal{X}_{\epsilon}}\le C\epsilon,\quad 1\le j\le k,
\end{align*}
with $\|\sigma_1 \mcal{L}_{g}h_j\|_{n,\partial  \mcal{X}_{\epsilon}}\le C\epsilon$ for $1\le j\le $ since all $h_{j}$ at least belongs to $C^{3}(\mcal{X})$.

Putting the bounds in the interior and near the boundary together, we conclude: with probability at least $1-2kne^{-Cn\epsilon^{d+4}}$,
\begin{align*}
    \|L_{n,\epsilon}h_{j}-\sigma_1 \mcal{L}_{g}h_{j}\|_{n}\le C\epsilon,\quad 1\le j\le k.
\end{align*}

Now, combining the above result with \eqref{eigenvfclose}, it follows that with probability at least $1-2kne^{-Cn\epsilon^{d+4}}-Cne^{-cn\theta^{2}\tilde{\delta}^{d}}$, we can find an orthonormal set $\tilde{h}_{1},\ldots,\tilde{h}_{k}$ of spanning $S$ such that 
\begin{align*}
    \|h_{j}-\tilde{h}_{j}\|_{n}\le C\epsilon.
\end{align*}
Here, recall that $\{h_{j}\}_{j=1}^{k}$ is an orthonormal basis for the eigenspace of eigenfunctions of $\mcal{L}_{g}$ with respect to eigenvalue $\Lambda=\Lambda_{l}$ and $\{\tilde{h}_{j}\}_{j=1}^{k}$ is an orthonormal set spanning $S$, i.e., the eigenspace of eigenbvectors of $L_{n,\epsilon}$ with respect to the eigenvalue $\lambda=\lambda_{l}$. These two sets of functions/vectors are close in $\|\cdot\|_{n}$ norm by $C\epsilon$. Therefore, with $\phi_i$ as the projection $i$-th eigenfunction into $\mbb{R}^{n}$ via the transportation map $\tilde{T}$ defined in \citet[Proposition 3]{green2021minimaxsmoothing}, we have: with probability at least $1-2kne^{-Cn\epsilon^{d+4}}-Cne^{-cn\theta^{2}\tilde{\delta}^{d}}$,
\begin{align*}
    \|v_{i}-\phi_i\|_{n}\le C\epsilon.
\end{align*}
Then, plugging the above approximation in \eqref{discreteseminorm} with \citet[Proposition 4.21]{weihs2023consistency}, we obtain:
\begin{align}\label{seminormsbound}
    \langle L_{n,\epsilon}^{s}f,f\rangle_{n}&=\sum_{i=1}^{n}\lambda_{i}^s\langle f,v_{i}\rangle_{n}^2\lesssim \sum_{i=1}^{n}\lambda_{i}^s\langle f,v_{i}-\phi_{i}\rangle_{n}^2+\sum_{i=1}^{n}\lambda_{i}^s\langle f,\phi_{i}\rangle_{n}^2\nonumber\\&\lesssim C\left(\epsilon+\sum_{i=1}^{n}\lambda_{i}^s\langle f,\phi_{i}\rangle_{n}^2\right).
\end{align}
Now recall \eqref{spectrallynorm}. For $n$ large enough (or equivalently $\epsilon$ small enough) we have
\begin{align*}
\langle L_{n,\epsilon}^{s}f,f\rangle_{n}& \lesssim C \sum_{i=1}^{n}\Lambda_{i}^s\langle f,\phi_{i}\rangle_{n}^2\le CM^{2}.
\end{align*}

with probability at least $1-Cn^2e^{-Cn\epsilon^{d+4}}-Cne^{-Cn}$, where we have used \eqref{rougheigenvalueapprox} above.

Furthermore, according to Lemma \ref{eigenvaluebound}, we have for $0<s<1$,
\begin{align}\label{eigenvaluesbound}
    \left(k^{\frac{2s}{d}}\wedge r^{-2s}\right)\lesssim \lambda_{k}^{s}\lesssim \left(k^{\frac{2s}{d}}\wedge r^{-2s}\right),
\end{align}
for $1\le k\le n$ (since the case $k=1$ can be bounded alone), with probability at least $1-Cne^{-Cn\epsilon^{d}}$.

Now, we are ready to proceed based on \eqref{biasvariancedecomp}:
\begin{align*}
    \|\hat{f}-f\|_{n}^{2}\le \frac{\langle L_{n,\epsilon}^{s}f,f\rangle_{n}}{\lambda_{K+1}^{s}}+\frac{K}{n},
\end{align*}
with probability at least $1-e^{-K}$ if $1\le K\le n$. According to \eqref{seminormsbound} and \eqref{eigenvaluesbound}, we have with probability at least $1-Cn^2e^{-Cn\epsilon^{d+4}}-Cne^{-Cn}-Cne^{-cn\epsilon^{d}}-e^{-K}$ and $n$ large enough:
\begin{align*}
    \|\hat{f}-f\|_{n}^{2}\lesssim \frac{M^{2}}{(K+1)^{2s/d} \wedge \epsilon^{-2s}}+\frac{K}{n}.
\end{align*}
Furthermore, based on the assumption $\epsilon\lesssim K^{-1/d}$, the above inequality becomes:
\begin{align}\label{tradeoffres}
    \|\hat{f}-f\|_{n}^{2}\lesssim M^{2}(K+1)^{-2s/d}+\frac{K}{n}.
\end{align}
By balancing the two terms on the right-hand side, we pick $K=\lfloor M^2n \rfloor^{d/(2s+d)}$. Then, it yields that
\begin{align}\label{finallyminimaxopt}
    \|\hat{f}-f\|_{n}^{2}\lesssim M^2(M^2n)^{-2s/(2s+d)},
\end{align}
with probability at least $1-Cn^2e^{-Cn\epsilon^{d+4}}-Cne^{-Cn}-Cne^{-cn\epsilon^{d}}-e^{-K}$.

If $M^2<n^{-1}$, we can take $K=1$ and obtain from \eqref{tradeoffres} that:
\begin{align*}
\|\hat{f}-f\|_{n}^{2}\lesssim \frac{1}{n}.
\end{align*}
If $M>n^{s/d}$, we take $K=n$ and in this case, we actually have $\hat{f}(X_i)=Y_i$ for $i=1,\ldots,n$ and
\begin{align*}
    \|\hat{f}-f\|_{n}^{2}=\frac{1}{n}\sum_{i=1}^{n}\varepsilon_{i}^{2}\lesssim C,
\end{align*}
with probability at least $1-e^{-n}$ for some constant $C$. Combining all above cases depending on choices of $K$, it yields that bound in Theorem \ref{mainthm}.

\end{proof}

\section{Proof of Theorem \ref{thm:lowerbound}}
We will first present an auxiliary lemma below from \citet[Lemma 3.2]{gyorfi2002distribution}.
\begin{lemma}\label{bayes_error}
    Let $u\in \mbb{R}^{l}$, for $l\in \mbb{N}$, and $\mbf{c}$ be a zero mean random variable taking values in $\{-1,1\}$. Moreover, denote by $\mbf{N}$ the  $l$-dimensional standard normal random variable independent of $\mbf{c}$. Set
    \begin{align*}
        \mbf{z}=\mbf{c}u+\mbf{N}.
    \end{align*}
    Then the error probability of the Bayes decision for $\mbf{c}$ based on $\mbf{z}$ is
    \begin{align*}
        \min_{\mcal{G}:\mbb{R}^{l}\rightarrow \mbb{R}}\mbb{P}(\mcal{G}(\mbf{z})\neq \mbf{c})=\Phi(-\|u\|),
    \end{align*}
    where $\Phi(\cdot)$ is the standard normal distribution function.
\end{lemma}

\begin{proof}[Proof of Theorem \ref{thm:lowerbound}]
We will mainly modify \citet[Proof of Theorem 3.2]{gyorfi2002distribution} for our fractional Sobolev space $H^s(\mcal{X},M)$, $0<s<1$. According to Assumption (A1) in Section \ref{sec:assump}, without loss of generality, we can consider $\mcal{X}=(0,1)^d$. Set $$r_n:=\lceil (M^2n)^{\frac{1}{2s+d}}\rceil.$$ We partition $\mcal{X}=(0,1)^d$ by $r_n^d$ cubes denoted by $\{A_{n,j}\}_{j=1}^{r_n^d}$ of side length $r_n^{-1}$ and with centers $\{a_{n,j}\}_{j=1}^{r_n^d}$. Choose a function $\bar{\psi}:\mbb{R}^d\rightarrow \mbb{R}$ such that its support is a subset of $[-\frac{1}{2},\frac{1}{2}]^d$, $\int \bar{\psi}^2(x)dx>0$, and $\bar{\psi}\in H^s(\mcal{X};1)$. Define $\psi:\mbb{R}^d\rightarrow \mbb{R}$ by $\psi(x):=M\cdot \bar{\psi}(x)$. It can be readily verified that
\begin{itemize}
    \item the support of $\psi$ is also a subset of $[-\frac{1}{2},\frac{1}{2}]^d$;
    \item $\int \psi^2(x)dx=M^2\cdot \int \bar{\psi}^2(x)dx>0$;
    \item  $\psi\in H^s(\mcal{X},M)$.
\end{itemize} 
The class of regression functions is indexed by a vector
\begin{align*}
    c_n=(c_{n,1},\ldots,c_{n,r_n^d})
\end{align*}
consisting of $\pm 1$ components so that `worst regression function' will depend on the sample size $n$. Let $\mcal{C}_n$ represent the set of all such vectors. Then, for each vector $c_n=(c_{n,1},\ldots,c_{n,r_n^d})\in \mcal{C}_n$, it corresponds to a function
\begin{align*}
    f^{(c_n)}(x):=\sum_{j=1}^{r_n^d}c_{n,j}\psi_{n,j}(x),
\end{align*}
where $\psi_{n,j}(x)=r_n^{-s}\psi(r_n(x-a_{n,j}))$. Then, if $x,y\in A_{n,i}$ for some $i$, it holds that
\begin{align*}
    |f^{(c_n)}(x)-f^{(c_n)}(y)|^2=|c_{n,i}|^2|\psi_{n,i}(x)-\psi_{n,i}(y)|^2=r_n^{-2s}|\psi(r_n(x-a_{n,i}))-\psi(r_n(y-a_{n,i}))|^2.
\end{align*}
Moreover, by definition,
\begin{align*}
    \int\int \frac{|\psi(r_n(x-a_{n,i}))-\psi(r_n(y-a_{n,i}))|^2}{\|r_n(x-y)\|^{2s+d}}r_n^{2d} dxdy\le M^2.
\end{align*}
It implies
\begin{align*}
    \int\int \frac{|f^{(c_n)}(x)-f^{(c_n)}(y)|^2}{\|x-y\|^{2s+d}}dxdy\le M^2.
\end{align*}
If $x\in A_{n,i}$ and $y\in A_{n,j}$ for $i\neq j$, i.e., $x$ and $y$ are in two disjoint supports, we can apply Jensen's inequality:
\begin{align*}
    |f^{(c_n)}(x)-f^{(c_n)}(y)|^2\le 3(|f^{(c_n)}(x)-f^{(c_n)}(\bar{x})|^2+|f^{(c_n)}(y)-f^{(c_n)}(\bar{y})|^2+|f^{(c_n)}(\bar{x})-f^{(c_n)}(\bar{y})|^2),
\end{align*}
where $\bar{x},\bar{y}$ are on the line between $x,y$ such that $\bar{x}$ is on the boundary of $A_{n,i}$ and $\bar{y}$ is on the boundary of $A_{n,j}$ and $f^{(c_n)}(\bar{x})=f^{(c_n)}(\bar{y})=0$ (because $\psi_{n,i}(\bar{x})=\psi_{n,j}(\bar{y})=0$). Then, we also have 
\begin{align*}
    \int\int \frac{|f^{(c_n)}(x)-f^{(c_n)}(y)|^2}{|x-y|^{2s+d}}dxdy\le M^2.
\end{align*}
Together, it shows that $f^{(c_n)}(x)\in H^s(\mcal{X};M)$.

Then, the minimix lower bound can be derived by showing the following lower bound:
\begin{align*}
    \underset{n\rightarrow \infty}{\lim\inf} \inf_{\hat{f}_n}\sup_{f^{(c_n)},c_n\in\mcal{C}_n}~\frac{r_n^{2s}}{M^{2}}\mbb{E}\|\hat{f}_n-f\|^2\ge C_1>0.
\end{align*}
Let $\hat{f}_n$ be an arbitrary estimate. Denote by $\hat{f}_{n,\psi}$ the projection of $\hat{f}_n$ onto $\{\psi_{n,j}\}$:
\begin{align*}
    \hat{f}_{n,\psi}=\sum_{j=1}^{r_n^d}\hat{c}_{n,j}\psi_{n,j}(x),
\end{align*}
where
\begin{align*}
    \hat{c}_{n,j}=\frac{\int_{A_{n,j}}\hat{f}_n(x)\psi_{n,j}(x)dx}{\int_{A_{n,j}}\psi^2_{n,j}(x)dx}.
\end{align*}
Then, we have:
\begin{align*}
    \|\hat{f}_n-f^{(c_n)}\|^2&\ge \|\hat{f}_{n,\psi}-f^{(c_n)}\|^2\\
    &=\sum_{j=1}^{r_n^d}\int_{A_{n,j}}(\hat{c}_{n,j}-c_{n,j})^2\psi_{n,j}^2(x)dx\\
    &=\int \psi^2(x)dx\cdot\sum_{j=1}^{r_n^d}(\hat{c}_{n,j}-c_{n,j})^2\frac{1}{r_n^{2s+d}}.
\end{align*}
Let $\tilde{c}_{n,j}$ be $1$ if $\hat{c}_{n,j}\ge 0$ and $-1$ otherwise. Noting that $|\hat{c}_{n,j}-c_{n,j}|\ge |\tilde{c}_{n,j}-c_{n,j}|/2$, we obtain:
\begin{align*}
    \|\hat{f}_n-f^{(c_n)}\|^2&\ge\int \psi^2(x)dx\cdot\frac{1}{4}\sum_{j=1}^{r_n^d}(\tilde{c}_{n,j}-c_{n,j})^2\frac{1}{r_n^{2s+d}}\\
    &\ge \int \psi^2(x)dx\cdot\frac{1}{r_n^{2s+d}}\sum_{j=1}^{r_n^d}\mbf{1}_{\tilde{c}_{n,j}\neq c_{n,j}}\\
    &=\frac{M^2}{r_n^{2s}}\int \bar{\psi}^2(x)dx\cdot\frac{1}{r_n^d}\sum_{j=1}^{r_n^d}\mbf{1}_{\tilde{c}_{n,j}\neq c_{n,j}}.
\end{align*}
Hence, it suffices to prove
\begin{align*}
    \underset{n\rightarrow \infty}{\lim\inf} \inf_{\tilde{c}_n}\sup_{c_n}~\frac{1}{r_n^d}\sum_{j=1}^{r_n^d}\mbb{P}(\tilde{c}_{n,j}\neq c_{n,j})>0.
\end{align*}
Now we randomize $c_n$. Let $\mbf{c}_{n,1},\ldots,\mbf{c}_{n,r_n^d}$ be a sequence of i.i.d. random variables independent of everything else such that
\begin{align*}
    \mbb{P}(\mbf{c}_{n,1}=1)=\mbb{P}(\mbf{c}_{n,1}=-1)=\frac{1}{2}.
\end{align*}
Let $\mbf{c}_n=(\mbf{c}_{n,1},\ldots,\mbf{c}_{n,r_n^d})$. Then, it holds that
\begin{align*}
    \underset{n\rightarrow \infty}{\lim\inf} \inf_{\tilde{c}_n}\sup_{c_n}~\frac{1}{r_n^d}\sum_{j=1}^{r_n^d}\mbb{P}(\tilde{c}_{n,j}\neq c_{n,j})\ge \inf_{\tilde{c}_n}~\frac{1}{r_n^d}\sum_{j=1}^{r_n^d}\mbb{P}( \mbf{c}_{n,j}\neq \tilde{c}_{n,j}).
\end{align*}
Here, we can view $\tilde{c}_{n,j}$ as a decision on $\mbf{c}_{n,j}$ based on $D_n=\{(X_i,Y_i)\}_{i=1}^{n}$. Its error is minimal for the Bayes decision $\bar{\mbf{c}}_{n,j}$, which is $1$ if $\mbb{P}(\mbf{c}_{n,j}=1|D_n)\ge \frac{1}{2}$ and $-1$ otherwise. Then, it yields that
\begin{align*}
    \inf_{\tilde{c}_n}~\frac{1}{r_n^d}\sum_{j=1}^{r_n^d}\mbb{P}( \mbf{c}_{n,j}\neq \tilde{c}_{n,j})&\ge \frac{1}{r_n^d}\sum_{j=1}^{r_n^d}\mbb{P}( \mbf{c}_{n,j}\neq \bar{\mbf{c}}_{n,j})\\
    &=\mbb{P}( \mbf{c}_{n,1}\neq \bar{\mbf{c}}_{n,1})\\
    &=\mbb{E}\left(\mbb{P}(\mbf{c}_{n,1}\neq \bar{\mbf{c}}_{n,1}|X_1,\ldots,X_n)\right).
\end{align*}
Note that for $X_i\in A_{n,1}$,
\begin{align*}
    Y_i=\mbf{c}_{n,1}\psi_{n,1}(X_i)+\varepsilon_i.
\end{align*}
Therefore, according to Lemma \ref{bayes_error}, the error probability of the Bayes decision $\bar{\mbf{c}}_{n,j}$ above satisfies:
\begin{align*}
    \mbb{P}(\mbf{c}_{n,1}\neq \bar{\mbf{c}}_{n,1}|X_1,\ldots,X_n)=\Phi\left(-\sqrt{\sum_{i=1}^{n}\psi_{n,1}^2(X_i)}\right),
\end{align*}
where $\Phi(\cdot)$ is the standard normal distribution function. Since $x\mapsto \Phi(-\sqrt{x})$ is convex, applying Jensen's inequality yields that
\begin{align*}
    \mbb{E}\left(\mbb{P}(\mbf{c}_{n,1}\neq \bar{\mbf{c}}_{n,1}|X_1,\ldots,X_n)\right)&\ge \Phi\left(-\sqrt{\mbb{E}\sum_{i=1}^{n}\psi_{n,1}^2(X_i)}\right)\\
    &=\Phi\left(-\sqrt{n\mbb{E}\psi_{n,1}^2(X_1)}\right)\\
    &=\Phi\left(-\sqrt{nr_n^{-2s+d}\int \psi^2(x)dx}\right)\\
    &\ge \Phi\left(-\sqrt{\int \bar{\psi}^2(x)dx}\right)>0.
\end{align*}
We then obtain the proof for the desired bound.
\end{proof}

\end{document}